\documentclass[10pt]{amsart}
\usepackage{fancyhdr}
\usepackage[english]{babel}
\usepackage{amssymb}
\usepackage{mathrsfs}
\usepackage{enumerate}
\usepackage{color}
\usepackage{ifpdf}
\usepackage{tikz}
\usepackage{tikz-cd}
\usetikzlibrary{decorations.pathmorphing,arrows}
\usepackage[initials]{amsrefs}
\usepackage{hyperref}

\usepackage{todonotes}

\usepackage{color}

\newtheorem{theorem}{Theorem}[section]
\newtheorem{thm}[theorem]{Theorem}
\newtheorem{lemma}[theorem]{Lemma}

\newtheorem{cor}[theorem]{Corollary}

\newtheorem{prop}[theorem]{Proposition}

\theoremstyle{definition}

\newtheorem*{ques*}{Question}

\numberwithin{equation}{section}

\title{Characterising acylindrical hyperbolicity via permutation actions}

\author[U. Bader]{Uri Bader}
\address{Weizmann Institute of Science, Israel}
\email{uri.bader@gmail.com}

\author[A. Sisto]{Alessandro Sisto}
	\address{Maxwell Institute and Department of Mathematics, Heriot-Watt University, Edinburgh, UK}
	\email{a.sisto@hw.ac.uk}

\begin{document}

\begin{abstract}
    We characterise acylindrical hyperbolicity of a group in terms of properties of an action of the group on a set (without any extra structure). In particular, this applies to the action of the group on itself by left multiplication, as well as the action on a (full measure subset of the) Furstenberg-Poisson boundary.
\end{abstract}

\maketitle

\section{Introduction}

The goal of this note is to characterise acylindrical hyperbolicity of a group $G$ in terms of the action of $G$ on itself, or on some other set. We do so by defining the notion of \emph{axial element} for an action (on a set) which is inspired by the notion of WPD element for an action on a hyperbolic space, as first considered in \cite{BF:WPD}.

Consider an action of the group $G$ on the non-empty set $X$ and fix an element $g\in G$ and a fundamental domain $D$ for the $\langle g \rangle$-action on $X$.
A subset of $X$ is called \emph{bounded} if it is contained in a finite union of translates of $D$ by elements of $\langle g \rangle$ and it is called \emph{unbounded} otherwise.
An element of $G$ is called \emph{tame} if it translates every bounded subset of $X$ into a bounded subset.
The pair $(g,D)$ is said to be an \emph{axial pair} for the action of $G$ on $X$ if the following are satisfied.
\begin{description}
    \item[Axiom 1] The subset of tame elements $t \in G$ such that $tD\cap D\neq \emptyset$ is finite.
    \item[Axiom 2] For every $h\in G$ there exists a bounded subset $B\subset X$ such that for every element $w \in G$ for which $hwD$ is unbounded, $X=wB\cup g^nB$ for some $n\in\mathbb{Z}$.
\end{description}
The element $g\in G$ is said to be an \emph{axial element for the action of $G$ on $X$} if there exists a fundamental domain $D\subset X$ such that the pair $(g,D)$ is an axial pair
and it is said to be an \emph{axial element} in $G$ if it is axial for the left regular action of $G$.


Recall that $g\in G$ is said to be a \emph{generalised loxodromic} element
if there is an acylindrical action of $G$ on a hyperbolic space where $g$ acts loxodromically.
Our main theorem gives equivalent characterisations of generalised loxodromicity. We refer to a probability measure as admissible if the semigroup generated by its support is the whole group.

\begin{thm}
\label{thm:main}
Let $G$ be a finitely generated, non-virtually cyclic group, and let $g\in G$. Then the following are equivalent.
\begin{enumerate}
    \item $g$ is generalised loxodromic.
    \item $g$ is axial for some action of $G$ on a set.
    \item $g$ is an axial element in $G$.
    \item $g$ is axial for the action of $G$ on some full-measure subset of its Furstenberg-Poisson boundary for some (any) admissible measure, with respect to a measurable fundamental domain.
    \end{enumerate}
\end{thm}

By definition, a non-virtually cyclic group is acylindrically hyperbolic if and only if it contains a generalised loxodromic element (see \cite{Osin} for various characterisations of acylindrical hyperbolicity). We thus immediately get:

\begin{cor}
Let $G$ be a finitely generated, non-virtually cyclic group. Then the following are equivalent:
\begin{enumerate}
    \item $G$ is acylindrically hyperbolic.
    \item $G$ has an action on some set that admits an axial element.
    \item $G$ contains an axial element.
    \item $G$ contains an axial element for its action on some full-measure subset of its Furstenberg-Poisson boundary for some (any) admissible measure, with respect to a measurable fundamental domain.
\end{enumerate}
\end{cor}

\subsection*{Outline}
In Section \ref{sec:axial} we study axial elements, proving various properties reminiscent of those of loxodromic WPD elements. These allow us to use the Bestvina-Bromberg-Fujiwara construction \cite{BBF} to construct hyperbolic spaces on which axial elements are loxodromic, see Theorem \ref{thm:axial_to_lox}. This shows that any of items (2),(3),(4) of Theorem \ref{thm:main} implies item (1).

In Section \ref{sec:gen_lox} we show Proposition \ref{prop:lox_to_axial} which says that generalised loxodromic elements are axial for various actions, proving that item (1) of Theorem \ref{thm:main} implies items (2),(3),(4). We note that Proposition \ref{prop:lox_to_axial} also says that the loxodromic elements for a specific acylindrical action on a hyperbolic space $Z$ are each axial for the action of $G$ on $Z$ (regarded as a set), as well as for the action on a natural subset of the boundary.



\section{Axial elements are generalised loxodromics}
\label{sec:axial}

\subsection{Heuristics}

In this subsection we briefly explain some of the heuristics behind the definition of axial element. One setup that one might think of is a group $G$ acting on a hyperbolic space $X$, which we want to only regard as a set, with $g$ being a loxodromic WPD. The fundamental domain $D$ that we would like to consider is, roughly, the set of points that closest-point project to an axis of $g$ to a fixed fundamental domain there. Now, it is known that $g$ is contained in a unique maximal virtually cyclic subgroup $E(g)$, and the definition of tame element is designed to capture the elements of said subgroup. Axiom (1) is the interpretation of $E(g)$ being virtually cyclic. Axiom (2) has a more involved relationship with properties of $g$ (and indeed, it will take some work to prove it in Proposition \ref{prop:lox_to_axial} below). As it will become clearer below, it simultaneously capture two facts. First, it captures the fact that translates of an axis of $g$ by elements outside of $E(g)$ have uniformly bounded projection to the axis. Secondly, it captures the fact that, given two translates of an axis of $g$, there are only finitely many translates "in between" the two.

\subsection{Properties of axial elements}

In this section we fix an action of the group $G$ on a non-empty set $X$.
We note that if $G$ is finite then any of its elements is axial in a trivial way.
We thus assume that $G$ is infinite.
We fix once and for all an axial pair $(g,D)$.

It follows from Axiom (1) that $g$ is of infinite order (indeed, if not then every element of $G$ would be tame, and every tame element $t$ can be multiplied on the left by a power of $g$ to obtain an element as in Axiom (1)).
We denote $D_i=g^iD_0$, in particular $D_0=D$, and get a partition of $X$,
$X=\coprod_{i\in \mathbb{Z}} D_i$.
We let an \emph{interval} be a subset of $X$ of the form $D_{[a,b]}=\cup_{i=a}^b D_i$,
for $a\leq b$, and note that a subset of $X$ is bounded if and only if it is contained in some interval.
We rewrite axiom (2) as follows: 

\medskip
(2') For every $h\in G$ there exists $m\in\mathbb{Z}$ such that for every element $w \in G$ for which $hwD$ is unbounded we have $X=wD_{[-m,m]}\cup g^nD_{[-m,m]}$ for some $n\in\mathbb{Z}$.
\medskip

Given $h\in G$ we let $m(h)$ be the minimal $m\in \mathbb{N}$ satisfying (2').
We note that for a tame element $t\in G$ we have $m(th)=m(h)$, since then $hwD$ is unbounded if and only if $thwD$ is unbounded.
We observe that $m:G\to \mathbb{N}$ is subadditive.

\begin{lemma}
For $h_1,h_2\in G$, $m(h_1h_2)\leq m(h_1)+m(h_2)$.
\end{lemma}

\begin{proof}
We denote $m_i=m(h_i)$. 
Fix $w \in G$ for which $h_1h_2wD$ is unbounded.
Then (using (2') with $h_1$ and $h_2w$ replacing $h$ and $w$) we have that $h_2wD_{[-m_1,m_1]}$ contains $X-g^{n_1}D_{[-m_1,m_1]}$ for some $n_1\in\mathbb{Z}$
and in particular it is unbounded.
It follows that there exists $j\in [-m_1,m_1]$ such that $h_2wD_j=h_2wg^jD$ is unbounded.
Then $wg^jD_{[-m_2,m_2]}$ contains $X-g^{n_2}D_{[-m_2,m_2]}$ for some $n_2\in\mathbb{Z}$.
Writing $m=m_1+m_2$, we deduce that 
\[ wD_{[-m,m]} \supseteq wg^jD_{[-m_2,m_2]} \supseteq X-g^{n_2}D_{[-m_2,m_2]}\supseteq X-g^{n_2}D_{[-m,m]}, \]
thus indeed $m(h_1h_2)\leq m=m_1+m_2$.
\end{proof}

We denote by $T\subseteq G$ the subset of tame elements (we will see soon that this is a subgroup)
and by $F\subset T$ its finite subset $\{t\in T \mid tD\cap D\neq \emptyset\}$.
We denote $W=G-T$ and call the elements of $W$ \emph{wild} elements.
We denote $M=m(e)$. 

The following lemma, in the setup of the heuristic discussion above, is best interpreted as saying that the $w^{-1}$-translate of the axis of $g$ has bounded projection to the axis of $g$.

\begin{lemma} \label{lem:h=e}
For $w\in W$ there exists $i\in \mathbb{Z}$ such that $wD_i$ is unbounded.
For every such $i$, $wD_{[i-M,i+M]}$ contains $X-D_{[a,b]}$ for some $a,b\in\mathbb{Z}$ 
with $b-a= 2M$. 
\end{lemma}

\begin{proof}
As in the previous proof, if the $w$-image of an interval is unbounded, then the $w$-image of one of the fundamental domains that it contains is unbounded as well. The second part of the statement is Axiom (2') for the case $h=e$, applied for $wg^i$, where $a=k-M$ and $b=k+M$.
\end{proof}

A subset of $X$ is called \emph{cobounded} if its complement is bounded.

\begin{lemma} \label{lem:wildiswild}
An element $w\in G$ is wild if and only if there exists a bounded subset $B\subset X$ such that $wB$ is cobounded.
\end{lemma}

\begin{proof}
The "if" part is clear and the "only if" part follows from Lemma~\ref{lem:h=e}
by setting $B=D_{[i-M,i+M]}$.
\end{proof}

The following lemma describes the same virtually cyclic subgroup $E(g)$ that was mentioned in the heursitic discussion above.

\begin{lemma}
\label{lem:virt_cyclic}
The subset $T\subset G$ is a virtually cyclic subgroup and it coincides with the commensurating subgroup of the subgroup $\langle g \rangle$ in $G$.
\end{lemma}

\begin{proof}
It is clear that $T$ is closed under taking products.
By Lemma~\ref{lem:wildiswild}, $W$ is closed under taking inverses:
for $w\in W$ and bounded $B\subset X$ such that $wB$ is cobounded,
$w^{-1}(X-wB)$ is unbounded, thus $w^{-1}\in W$.
It follows that $T$ is also closed under taking inverses, thus it is a subgroup of $G$.

From Axiom (1) and the fact that $g\in T$
we get that for every $t\in T$ there exists $n\in \mathbb{Z}$ such that $g^nt$ belongs to the finite subset $F$.
It follows that $T$ is virtually cyclic.
Since $T$ is virtually cyclic and $g$ is an element of infinite order, we get that $T$ normalizes a finite index subgroup of $\langle g \rangle$. In particular, $T$ commensurates $\langle g \rangle$.

We are left to show that the commensurating subgroup of $\langle g \rangle$ is contained in $T$.
To see this, we assume by contradiction that $w$ is a wild element
which commensurates $\langle g \rangle$.
We thus have, by Lemma~\ref{lem:wildiswild}, a bounded subset $B\subset X$ such that $wB$ is cobounded and $a,b\in \mathbb{N}$ such that $wg^aw^{-1}=g^b$.
Then the set $wB$ contains the sets $D_{nb}$ for infinitely many values of $n\in\mathbb{Z}$,
but for every such $n$,
$B\supseteq w^{-1} D_{nb}=w^{-1}g^{bn}D=g^{an}w^{-1}D$
and the union of the sets $g^{an}w^{-1}D$ is unbounded. This contradiction finishes the proof.
\end{proof}

We conclude that $G$ is virtually cyclic if and only if $G=T$.
For the rest of the section we assume that $G$ is not virtually cyclic, thus $W\neq\emptyset$.
We note that $W$ is stable under taking inverses and multiplying by tame elements from either side, as $T$ is a group.

For every $w\in W$ there exists $i\in \mathbb{Z}$ such that $wD_i=wg^iD$ is unbounded.
We denote the set of such $i$'s by $I(w)$.
This is where $w$ gets really wild.
It turns out that these sets are uniformly bounded. Heuristically, $I(w)$ determines the projection of the $w^{-1}$-translate of the axis of $g$.

\begin{lemma} \label{lem:Ibounded}
For every $w\in W$, the diameter of $I(w)\subset \mathbb{Z}$ is bounded from above by $2M$.
\end{lemma}

\begin{proof}
Fix $w\in W$ and $i,j\in I(w)$.
By Lemma~\ref{lem:h=e}, the sets 
$wD_{[i-M,i+M]}$ and $wD_{[j-M,j+M]}$ are cobounded, thus intersect non-trivially.
It follows that the intervals $[i-M,i+M]$ and $[j-M,j+M]$
intersect non-trivially, thus $|i-j|\leq 2M$.
\end{proof}

We define the function $i:W\to \mathbb{Z}$ by
\[ i(w)=[(\min I(w)+\max I(w))/2]. \]
Note that $I(w)\subseteq [i(w)-M,i(w)+M]$ (though $i(w)$ might not lie in $I(w)$).
We think of $i(w)$ as the wilderness center of $w$.

The following lemmas build up to Lemma \ref{lem:ra}, which is the interpretation in our setup of the Behrstock inequality (Axiom (P1) of \cite{BBF}) in our context. Once one rewrites said inequality as we did in Lemma \ref{lem:ra}, it is natural to try to understand the quantities $|i(u)-i(v)|$:

\begin{lemma}
\label{lem:coarse_lip}
For every $u,v\in W$ we have $|i(u)-i(v)|\leq m(uv^{-1})+2M$.
\end{lemma}

\begin{proof}
Fix $i\in I(u)$. Then $|i-i(u)|\leq M$.
Set $m=m(uv^{-1})$.
Since $ug^iD=(uv^{-1})vg^iD$ is unbounded, we get that $vg^iD_{[-m,m]}$ is cobounded. 
It follows that for some $j\in [-m,m]$, $vg^iD_j=vg^{i+j}D$ is unbounded,
thus $i+j\in I(v)$ and therefore $|i+j-i(v)|\leq M$.
We conclude that $|i(u)-i(v)+j|\leq 2M$, thus indeed, $|i(u)-i(v)|\leq m+2M$.
\end{proof}

\begin{lemma} \label{lem:h=e2}
For every $w\in W$, $wD_{[i(w)-2M,i(w)+2M]}$ contains $X-D_{[a,b]}$ for some $a,b\in\mathbb{Z}$ 
with $b-a= 2M$.
\end{lemma}

\begin{proof} 
Follows immediately from Lemma~\ref{lem:h=e},
as for $i\in I(w)$, $D_{[i-M,i+M]}$ is contained in $D_{[i(w)-2M,i(w)+2M]}$.
\end{proof}

\begin{lemma} \label{lem:n+}
For every $w\in W$,
\[ wD_{[i(w)-4M,i(w)+4M]} \quad \mbox{contains} \quad X-D_{[i(w^{-1})-2M,i(w^{-1})+2M]}.  \]
\end{lemma}

\begin{proof}
By Lemma~\ref{lem:h=e2} and the fact that the inverse of a wild element is also wild, the sets
$wD_{[i(w)-2M,i(w)+2M]}$ and $w^{-1}D_{[i(w^{-1})-2M,i(w^{-1})+2M]}$ are cobounded.
In fact, $w^{-1}D_{[i(w^{-1})-2M,i(w^{-1})+2M]}$ contains $X-D_{[a,b]}$
for some $a,b\in \mathbb{Z}$ with $b-a=2M$.
Hence $wD_{[a,b]}$ contains $X-D_{[i(w^{-1})-2M,i(w^{-1})+2M]}$.
In particular, $wD_{[a,b]}$ is cobounded, thus it intersects non-trivially the cobounded set $wD_{[i(w)-2M,i(w)+2M]}$.
It follows that 
\[ [a,b] \cap [i(w)-2M,i(w)+2M] \neq \emptyset, \]
hence
\[ [a,b]\subset [i(w)-4M,i(w)+4M]. \]
We conclude that $wD_{[a,b]}$ is contained in $wD_{[i(w)-4M,i(w)+4M]}$,
thus
\[ X-D_{[i(w^{-1})-2M,i(w^{-1}+2M)]} \subseteq wD_{[a,b]} \subseteq wD_{[i(w)-4M,i(w)+4M]}. \]
\end{proof}

As mentioned above, the following lemma will give us one the axioms from \cite{BBF}.

\begin{lemma} \label{lem:ra}
Assume $u,v\in W$ are such that also $u^{-1}v\in W$.
Then 
\[ \min \{|i(u)-i(v^{-1}u)|,|i(v)-i(u^{-1}v)|\} \leq 5M. \]
\end{lemma}

\begin{proof}
Assume $|i(u)-i(v^{-1}u)|>4M$.
Then 
\[ [i(u)-2M,i(u)+2M] \cap [i(v^{-1}u)-2M,i(v^{-1}u)+2M] =\emptyset \]
and we get by Lemma~\ref{lem:n+} (applied with $w=u^{-1}$ and $w=u^{-1}v$) that 
\[ X=u^{-1}D_{[i(u^{-1})-4M,i(u^{-1})+4M]} \cup u^{-1}vD_{[i(u^{-1}v)-4M,i(u^{-1}v)+4M]}. \]
Therefore, 
\[ X=D_{[i(u^{-1})-4M,i(u^{-1})+4M]} \cup vD_{[i(u^{-1}v)-4M,i(u^{-1}v)+4M]}. \]
It follows that $vD_{[i(u^{-1}v)-4M,i(u^{-1}v)+4M]}$ is unbounded, thus 
\[ [i(u^{-1}v)-4M,i(u^{-1}v)+4M]\cap [i(v)-M,i(v)+M] \neq \emptyset. \]
Therefore $|i(v)-i(u^{-1}v)|\leq 5M$.
\end{proof}

At this point, we would be able to show all axioms from \cite{BBF} except one. This remaining axiom says, roughly, that given two translates of an axis of $g$ there are only finitely many other translates onto which the two fixed translates project far away from each other.

It is convenient to consider $i:G\to \mathbb{Z}\cup \{\infty\}$ by setting $i(t)=\infty$ for $t\in T$.
Note that for every $h\in G$, $i(g^nh)=i(h)$ and $i(hg^n)=i(h)-n$,
with the usual conventions regarding $\infty$.
Given $h\in G$, we consider the function $w \mapsto i(hw)-i(w) \in \mathbb{Z}\cup \{\infty\}$,
which is well defined for $w\in W$.
Note also that $i(hwg^n)-i(wg^n)=i(hw)-i(w)$,
thus we get a well defined function
\[ f_h:W/\langle g \rangle \to \mathbb{Z}\cup \{\infty\}, \quad f_h(w\langle g \rangle)= i(hw)-i(w).\]

The next lemma will give the remaining axiom from \cite{BBF}. The proof uses the Behrstock inequality (Lemma \ref{lem:ra}) combined with the consequence of Axiom (2) that, roughly, the number of large projections under consideration is a Lipschitz function on the group $G$. 

\begin{lemma}
\label{lem:fin_large_proj}
Assume $G$ is finitely generated.
Then there exists $N\geq 0$ such that for every $h\in G$, $|f_h(w\langle g \rangle)|\leq N$ for all but finitely many cosets in $W/\langle g \rangle$.
\end{lemma}

\begin{proof}
Fix a finite generating set $S$ of $G$
and set $L=\max\{m(s)\mid s\in S\}+2M$. 
Set $N=20M+L$.
We fix $h\in G$, 
consider the set of exceptional cosets, 
\[ E=\{w\langle g \rangle\mid |f_h(w\langle g \rangle)|> N\} \subseteq W/\langle g \rangle\]
and prove the lemma by showing that $|E|\leq 2|h|\cdot [T:\langle g \rangle]$,
which is finite by Lemma~\ref{lem:virt_cyclic}.
Here $|h|$ denotes the word length of $h$ with respect to $S$.

We write $h=s_1\dots s_{|h|}$ for $s_k\in S$ and denote $h_k=s_k\dots s_{|h|}$.
The size of the set 
\[ E_0=\{h_k^{-1}t\langle g \rangle \mid k=1,\ldots,|h|,~t\in T\} \subseteq G/\langle g \rangle \]
is bounded by $|h|\cdot [T:\langle g \rangle]$.
Defining $E_1=E-E_0$, it is thus enough to show that $|E_1|\leq |h|\cdot [T:\langle g \rangle]$.
Letting $\tilde{E}_1\subset W$ be a set of representatives of left $T$-cosets of elements of $E_0$ in $W$,
we are left to show that $|\tilde{E}_1|\leq |h|$.

Note that for every pair of distinct elements $u,v\in \tilde{E}_1$, we have $u^{-1}v\in W$
and, by the fact that $E_1\cap E_0=\emptyset$,
for every $k=1,\ldots,|h|$ we have
$h_ku\in W$,
and thus by Lemma \ref{lem:coarse_lip},
\[ |i(h_ku)-i(h_{k-1}u)|\leq m(h_kh_{k-1}^{-1})=m(s_k) \leq L. \]
We conclude that for every $u\in \tilde{E}_1$ there exists $k(u)$ such that both
\[ |i(h_{k(u)}u)-i(u)|, |i(h_{k(u)}u-i(hu)|> 10M. \]
Indeed, the first integer $k$ such that $|i(h_ku)-i(u)|>10M$ will do, as we have $|i(hu)-i(u)|> N=20M+L$.
We claim that the resulting map 
\[ k:\tilde{E}_1 \to \{1,\dots, |h|\}, \quad u\mapsto k(u) \]
is injective.
Proving this claim will guarantee $|\tilde{E}_1|\leq |h|$, thus concluding the proof.

We now fix a pair of distinct elements $u,v\in \tilde{E}_1$ and argue to show that indeed $k(u)\neq k(v)$.
Without loss of the generality, using Lemma~\ref{lem:ra}, we assume that $|i(u)-i(v^{-1}u)|\leq 5M$.
Since $|i(h_{k(u)}u)-i(u)|> 10M$,
we conclude that $|i(h_{k(u)}u)-i(v^{-1}u)|> 5M$.
Applying again Lemma~\ref{lem:ra}, for $h_{k(u)}u$ and $h_{k(u)}v$, we get 
$|i(h_{k(u)}v)-i(u^{-1}v)|\leq 5M$.
Similarly, since $|i(hu)-i(u)|> N >10M$,
we get 
$|i(hu)-i(v^{-1}u)|> 5M$, thus $|i(hv)-i(u^{-1}v)|\leq 5M$.
Therefore we have $|i(h_{k(u)}v)-i(hv)|\leq 10M$.
Since $|i(h_{k(v)}v)-i(hv)|> 10M$ we must have $k(u)\neq k(v)$, as required.
This finishes the proof of the claim.
\end{proof}

Recall that $I(w)$ represents the wilderness zone of $w\in W$.
For precise estimations we replaced it by $i(w)$.
We think of $I(w^{-1})$ as the rough image in $\mathbb{Z}$ of $\infty$ under $w$.
Identifying $\mathbb{Z}$ with $\langle g \rangle$ via $n\mapsto g^n$,
we observe that $I:W\to \langle g \rangle$ is equivariant in the sense that
$I((wg)^{-1})=I(w^{-1})$ and $I((gw)^{-1})=g^{-1}I(w^{-1})$.
Next we turn it into a map which takes values in bounded subsets of $T$, rather than of $\langle g \rangle$,
maintaining the equivariance properties.
We set
\[ \pi:W \to 2^T, \quad \pi(w)=\bigcup_{s,t\in T} tI(sw^{-1}t), \]
thinking of $\pi(w)$ as the rough image in $T$ of $\infty$ under $w$. The following lemma summarises all properties of $\pi$, essentially saying that all axioms from \cite{BBF} are satisfied.

\begin{lemma} \label{lem:piproperties}
We fix a word distance $d$ on $T$.
The map $\pi:W \to 2^T$ satisfies the following properties:
\begin{itemize}
    \item $\pi$ is bi-$T$-invariant, that is for $w\in W$, $s,t\in T$, $\pi(swt)=\pi(w)$.
    \item There is a uniform bound on the diameters of the subsets $\pi(w)\subseteq T$, $w\in W$. 
    \item There is a uniform bound on the set of numbers 
    \[ \min\{d(\pi(u),\pi(vu^{-1}),d(\pi(v),\pi(uv^{-1}))\} \]
    which are obtained by running over all pairs $u,v\in W$ such that $uv^{-1}\in W$.
    \item There is a uniform constant $C>0$ such that for every $h\in G$,
    \[ |\{w\in W \cap Wh^{-1} \mid d(\pi(w),\pi(wh))>C\}|<\infty. \]
    \end{itemize}
\end{lemma}

\begin{proof}
That $\pi$ is bi-$T$-invariant follows at once from the equivariance properties of $I$.
It follows that $\pi(w)=\cup_{t\in F} tI(w^{-1}t) \subseteq F\cdot \cup_{t\in F}I(w^{-1}t)$, where $F\subset T$ is a set of cosets representatives for $\langle g \rangle$.
For every $w\in W$ we have $I(w)\subseteq [i(w)-M,i(w)+M]$, and hence it follows from Lemma~\ref{lem:coarse_lip} that  $\cup_{t\in F}I(w^{-1}t)\subseteq [i(w)-2M,i(w)+2M]$ (since $m(t)=m(e)=M$).
We conclude that the diameters of the sets $\pi(w)$ are uniformly bounded.
Similarly, the third item follows from Lemma~\ref{lem:ra} and the fourth item follows from Lemma~\ref{lem:fin_large_proj}.
\end{proof}

Finally, we combine all results in this section with \cite{BBF} to construct actions on hyperbolic spaces starting from axial elements.

\begin{thm}
\label{thm:axial_to_lox}
Let the finitely generated group $G$ act on the set $X$, and let $g\in G$ be axial. 
Assume $G$ is not virtually cyclic.
Then $G$ acts acylindrically on some quasi-tree $Z$ such that $g$ is a loxodromic WPD element for the action on $Z$.
\end{thm}

\begin{proof}
We will apply the Bestvina-Bromberg-Fujiwara construction \cite{BBF} to the collection of all cosets of $T$ in $G$ (each endowed with a word metric). To do so, we need to define projections $\pi_{h_1T}(h_2T)\subseteq h_1T$ for all $h_1T\neq h_2T$ satisfying suitable properties. The output of the construction is a quasi-tree on which $G$ acts and which contains isometrically embedded copies of all cosets of $T$, so that in particular $g$ is loxodromic, see \cite[Theorem B]{BBF}.

Given $h_1,h_2\in T$ with $h_1T\neq h_2 T$ we define projections equivariantly by
$$\pi_{h_1T}(h_2T)=h_1 \pi(h_1^{-1}h_2T).$$
Boundedness of projections is axiom (P0) of Bestvina-Bromberg-Fujiwara \cite[Section 1]{BBF}, and the remaining 2 axioms (P1),(P2) require that there exists a constant $B$ such that:
\begin{itemize}
    \item for all $h_1,h_2,h_3\in G$ such that the cosets $h_iT$ are pairwise distinct we have
    $$\min \{d_{h_1T}(\pi_{h_1T}(h_2T),\pi_{h_1T}(h_3T)),d_{h_2T}(\pi_{h_2T}(h_1T),\pi_{h_2T}(h_3T)) \leq B,$$ and
    \item for all $h_1,h_2\in G$ such that the cosets $h_iT$ are distinct we have
    $$|\{hT\ |\ d_{hT}(\pi_{hT}(h_1T),\pi_{hT}(h_2T))\geq B\}|<+\infty.$$
\end{itemize}
These axioms are satisfied by Lemma~\ref{lem:piproperties}.
We are now in a position to apply \cite[Theorem B]{BBF} and \cite[Theorem 6.4]{BBFS} (for the acylindricity conclusion) to obtain the required space.
\end{proof}

\section{Generalised loxodromics are axial}
\label{sec:gen_lox}

The aim of this section is to show that a generalised loxodromic is axial for various actions of the ambient group. We start with a lemma that says, roughly, that axiality can be "pulled-back".

\begin{lemma}
\label{lem:pull-back}
Let $G$ be a group acting on the sets $X,Y$, and let $f:X\to Y$ be a $G$-equivariant map. If the pair $(g,D)$ is axial for the action of $G$ on $Y$, then $(g,f^{-1}D)$ is axial for the action of $G$ on $X$.
\end{lemma}

\begin{proof}
Consider an axial pair $(g,D)$ for the action of $G$ on $Y$.
We will use the terminology, e.g., "$X$-bounded" or "$Y$-bounded" to clarify which action we are considering.

First of all, note that if $w\in G$ is $Y$-wild then it is also $X$-wild for. In fact, by Lemma \ref{lem:wildiswild} there exists a $Y$-bounded set $B$ such that $wB$ is $Y$-cobounded, and taking preimages we see that $f^{-1}B$ is $X$-bounded but $wf^{-1}(B)=f^{-1}(wB)$ is $X$-cobounded.

In particular, we see that the set of $X$-tame elements $t$ with $tf^{-1}D\cap f^{-1}D\neq\emptyset$ is contained in the set of $Y$-tame elements with $tD\cap D\neq\emptyset$, yielding Axiom (1).

Regarding Axiom (2), note that, given $h\in G$, the set of elements $w$ such that $hwD$ is $Y$-unbounded is contained in the set of elements $w$ such that $hwf^{-1}D$ is $X$-unbounded. Hence, given a $B$ as in Axiom (2) for the action on $Y$, the subset $f^{-1}(B)$ satisfies Axiom (2) for the action on $X$.
\end{proof}

We are now ready to show that generalised loxodromics are axial, for various actions.

\begin{prop}
\label{prop:lox_to_axial}
Let $G$ be a non-virtually cyclic group and let $g\in G$ be a loxodromic WPD for the action of $G$ on the hyperbolic space $Z$. Then $g$ is axial for:
\begin{enumerate}
\item The action of $G$ on $Z$ (regarded as a set).

\item The action of $G$ on itself by left multiplication.

\item The action of $G$ on $\partial Z-(G\cdot g^+\cup G\cdot g^-)$, where $\partial Z$ is the Gromov boundary and $g^\pm$ are the fixed points of $g$ in $\partial Z$.

    \item The action of $G$ on a full-measure set of the Furstenberg-Poisson boundary $\mathcal B$ associated with any given admissible measure $\mu$.
\end{enumerate}
 In the third and fourth case, the fundamental domain for $g$ can be taken to be a Borel set.
\end{prop}

\begin{proof}
We will show that $g$ is axial for the action on $X=Z\cup \partial Z-(G\cdot g^+\cup G\cdot g^-)$, with fundamental domain $D$ such that $D\cap (\partial Z-(G\cdot g^+\cup G\cdot g^-))$ is a Borel set. 
Once we do this, items (1) and (3) follow from Lemma \ref{lem:pull-back} applied to the obvious inclusions
and item (2) will follow by considering the orbit map $G\to X$ associated with an arbitrary element $x\in X$. 
Item (4) follows as well from Lemma \ref{lem:pull-back}, combined with the fact that there exists an equivariant measurable map $\mathcal B \to \partial Z-G\cdot g^+\cup G\cdot g^-$. The existence of an equivariant map $\mathcal B \to \partial Z$ is a consequence of \cite[Theorem 1.1]{MT}, but the statement of \cite[Theorem 2.3]{BCFS} (in view of \cite[Theorem 2.7]{BF-ICM}, and \cite[Remark 4]{GST} to drop the separability assumption) gives this more directly. To conclude, we only have to argue that the preimage of $G\cdot g^+\cup G\cdot g^-$ has measure $0$. This holds because otherwise that preimage would have full measure by ergodicity of $\mathcal B$, and hence there would be an equivariant map $\mathcal{B}\to G\cdot g^+\cup G\cdot g^-$ endowed with the discrete metric. Since $G$ is not virtually cyclic and $g$ is loxodromic WPD, the stabilisers of $g^\pm$ are not the whole $G$ (for example because the stabiliser of $\{g^\pm\}$ is the subgroup $E(g)$ used below). Hence, the action is fixed-point free, and this violates metric ergodicity, see \cite[Page 2; Theorem 2.7]{BF-ICM}.



We thus proceed to show that $g$ is axial for the above mentioned space $X$.
To avoid confusion between the metric notion of boundedness and being contained in finitely many fundamental domains, we will use the terminology "bounded in $X$" for the latter. We will use the terminology "unbounded in $X$" similarly.

Denote $\gamma=\langle g\rangle x_0$, for some fixed basepoint $x_0\in Z$ (we think of $\gamma$ as the axis of $g$). We consider a $g$-equivariant (closest point projection) map $\pi: Z\to \gamma$ that associates to $z\in Z$ a closest point in $\gamma$. In fact, we can extend $\pi$ to $\bar Z=Z\cup(\partial Z-\{g^\pm\}$ to a coarsely locally constant map, meaning that there exists a constant $C_0$ such that any point of $\bar Z$ has a neighborhood whose image under $\pi$ has diameter at most $C_0$. We can find a fundamental domain $D$ such that $\pi(D)$ is a bounded subset of $\gamma$ and such that any bounded subset of $\gamma$ is contained in finitely many fundamental domains, that is, it is bounded in $X$. Moreover, we can assume that the intersection of $D$ with $\partial Z-(G\cdot g^+\cup G\cdot g^-)$ is Borel.

 The WPD element $g$ is contained in a virtually cyclic subgroup $E(g)$ which has the property that there exists a constant $C_1$ such that for all $h\notin E(g)$ we have that $\pi(h\gamma)$ has diameter at most $C_1$, see e.g. \cite[Corollary 4.4]{Si:contr}. Up to enlarging $C_1$, for all $h \notin E(g)$ there are neighborhoods $N^\pm$ in $Z\cup \partial Z$ of $g^\pm$ such that $hN^\pm\cap \{g^\pm\}=\emptyset$ and $\pi(hN^\pm)$ has diameter at most $C_1$. This means that, given $h\notin E(g)$, all but finitely fundamental domains are mapped into a certain interval $D_{[-n,n]}$ for some $n$ (depending on $h$). Those finitely many fundamental domains are then necessarily mapped to an unbounded subset of $X$. In summary, if $h\notin E(g)$, then $h$ is not tame, and therefore tame elements all lie in $E(g)$. We can now show Axiom (1), by showing that any coset $\langle g\rangle t$ in $E(g)$ can only contain finitely many elements $t'$ with $t'D\cap D\neq \emptyset$. Indeed, since $t$ is tame we have that $tD$ is bounded in $X$, so that for any sufficiently large $n$ we have that $g^{\pm n}tD\cap D=\emptyset$, concluding the proof of Axiom (1).

To show Axiom (2), we will consider projections to translates of $\gamma$, which we denote $\pi_{k\gamma}:Z\cup (\partial Z-\{kg^\pm\}) \to k\gamma$ (with $\pi_\gamma=\pi$). Fix $h\in G$. We claim that there exists $C_3\geq 0$ (which is allowed to depend on $h$) such that if $hwD$ is unbounded in $X$ then $\pi_{w\gamma}(h^{-1}\gamma)$ lies within $C_3$ of $wx_0$. To prove the claim, note that it suffices to show that $\pi(w^{-1}h^{-1}g^+)$ or $\pi(w^{-1}h^{-1}g^-)$ project close to $x_0$, since $\pi$ is coarsely locally constant and $w^{-1}h^{-1}\gamma$ has bounded projection to $\gamma$ (note that $hw$ is wild, hence it does not lie in $E(g)$, and therefore the same is true of $w^{-1}h^{-1}$). Now, $hwD$ being unbounded in $X$ means that there are points $z$ in $Z\cup \partial Z$ arbitrarily close (in the topology of $Z\cup \partial Z$) to either $g^+$ or $g^-$ and such that $w^{-1}h^{-1}z\in D$, and in particular $w^{-1}h^{-1}z$ projects within bounded distance of $x_0$ in $\gamma$ (since $\pi(D)$ is a bounded subset of $\gamma$). Since $\pi$ is coarsely locally constant, $w^{-1}h^{-1}g^+$ or $w^{-1}h^{-1}g^-$ also project uniformly close to $x_0$ in $\gamma$.

We claim that there exists a bounded subset $B_0$ of $X$ and a constant $C_4$ such that for all wild elements $w$ there exists some $n$ such that $\pi(w(X-g^nB_0))$ lies within Hausdorff distance $C_4$ of $\pi(w\gamma)$. This is equivalent to finding $C_4$ and $B_0$ such that for all wild $w$ we have that $\pi_{w\gamma}(X-g^nB_0)$ lies within Hausdorff distance $C_4$ of $\pi_{w\gamma}(\gamma)$ for some $n$. We will use the fact that there exists a constant $C_5$ such that for all wild elements $w$ and $x\in X$ we have
$$\min\{d_Z(\pi(w\gamma),\pi(x)),d_Z(\pi_{w\gamma}(\gamma),\pi_{w\gamma}(x))\}\leq C_5;$$
this is essentially the fact that axes of WPD elements satisfy Axiom (P1) of \cite{BBF}, see e.g. the arguments of \cite[Lemma 2.5]{Si:contr}. Since $\pi(w\gamma)$ is bounded uniformly over all wild $w$, we can take $B_0$ such that for all wild $w$ there exists $n$ with the property that if $x\in X-g^nB_0$ then $d_Z(\pi(x),\pi(w\gamma))$ is larger than $C_5$. We then have that $\pi_{w\gamma}(x)$ lies within distance $C_5$ of $\pi_{w\gamma}(\gamma)$, as required.

On the other hand, for all $C\geq 0$ there is a bounded subset of $X$ that contain all points $x$ with $\pi(x)$ within $C$ of $x_0$. Hence, for any constant $C\geq0$ there exists a bounded subset $B\supseteq B_0$ of $X$ with the property that if $wB$ does not contain $X-g^nB$ for any $n$, then $\pi(w^{-1}(g^\pm))$ is $C$-far from $x_0$, and hence $\pi_{w\gamma}(\gamma)$ is $C$-far from $wx_0$. Hence, if for all bounded subsets $B$ of $X$ we had some $w\in G$ with $hwD$ unbounded but with $wB$ not containing any $X-g^nB$, then we would have elements $w$ such that the distances $d_Z(\pi_{w\gamma}(\gamma),\pi_{w\gamma}(h^{-1}\gamma))$ are arbitrarily large. But any such distance is bounded in terms of $d_Z(\gamma,h^{-1}\gamma)$ since closest point projections in a hyperbolic space are coarsely Lipschitz, a contradiction.
\end{proof}

\bibliographystyle{alpha}
\bibliography{biblio.bib}

\begin{bibdiv}
\begin{biblist}

\bib{BCFS}{article}{
      author={Bader, Uri},
      author={Caprace, Pierre-Emmanuel},
      author={Furman, Alex},
      author={Sisto, Alessandro},
       title={Hyperbolic actions of higher-rank lattices come from rank-one
  factors},
        date={2022},
     journal={arXiv preprint arXiv:2206.06431},
}

\bib{BF-ICM}{inproceedings}{
      author={Bader, Uri},
      author={Furman, Alex},
       title={Boundaries, rigidity of representations, and {L}yapunov
  exponents},
        date={2014},
   booktitle={Proceedings of the {I}nternational {C}ongress of
  {M}athematicians---{S}eoul 2014. {V}ol. {III}},
   publisher={Kyung Moon Sa, Seoul},
       pages={71\ndash 96},
      review={\MR{3729019}},
}

\bib{BBF}{article}{
      author={Bestvina, Mladen},
      author={Bromberg, Ken},
      author={Fujiwara, Koji},
       title={Constructing group actions on quasi-trees and applications to
  mapping class groups},
        date={2015},
        ISSN={0073-8301},
     journal={Publ. Math. Inst. Hautes \'{E}tudes Sci.},
      volume={122},
       pages={1\ndash 64},
         url={https://doi.org/10.1007/s10240-014-0067-4},
      review={\MR{3415065}},
}

\bib{BBFS}{article}{
      author={Bestvina, Mladen},
      author={Bromberg, Ken},
      author={Fujiwara, Koji},
      author={Sisto, Alessandro},
       title={Acylindrical actions on projection complexes},
        date={2019},
        ISSN={0013-8584},
     journal={Enseign. Math.},
      volume={65},
      number={1-2},
       pages={1\ndash 32},
         url={https://doi.org/10.4171/lem/65-1/2-1},
      review={\MR{4057354}},
}

\bib{BF:WPD}{article}{
      author={Bestvina, Mladen},
      author={Fujiwara, Koji},
       title={Bounded cohomology of subgroups of mapping class groups},
        date={2002},
        ISSN={1465-3060},
     journal={Geom. Topol.},
      volume={6},
       pages={69\ndash 89},
         url={https://doi.org/10.2140/gt.2002.6.69},
      review={\MR{1914565}},
}

\bib{GST}{article}{
      author={Gruber, Dominik},
      author={Sisto, Alessandro},
      author={Tessera, Romain},
       title={Random {G}romov's monsters do not act non-elementarily on
  hyperbolic spaces},
        date={2020},
        ISSN={0002-9939},
     journal={Proc. Amer. Math. Soc.},
      volume={148},
      number={7},
       pages={2773\ndash 2782},
         url={https://doi.org/10.1090/proc/14754},
      review={\MR{4099767}},
}

\bib{MT}{article}{
      author={Maher, Joseph},
      author={Tiozzo, Giulio},
       title={Random walks on weakly hyperbolic groups},
        date={2018},
        ISSN={0075-4102},
     journal={J. Reine Angew. Math.},
      volume={742},
       pages={187\ndash 239},
         url={https://doi-org.proxy.cc.uic.edu/10.1515/crelle-2015-0076},
      review={\MR{3849626}},
}

\bib{Osin}{article}{
      author={Osin, D.},
       title={Acylindrically hyperbolic groups},
        date={2016},
        ISSN={0002-9947},
     journal={Trans. Amer. Math. Soc.},
      volume={368},
      number={2},
       pages={851\ndash 888},
         url={https://doi.org/10.1090/tran/6343},
      review={\MR{3430352}},
}

\bib{Si:contr}{article}{
      author={Sisto, Alessandro},
       title={Contracting elements and random walks},
        date={2018},
        ISSN={0075-4102},
     journal={J. Reine Angew. Math.},
      volume={742},
       pages={79\ndash 114},
         url={https://doi.org/10.1515/crelle-2015-0093},
      review={\MR{3849623}},
}

\end{biblist}
\end{bibdiv}

\end{document}